\documentclass[12pt,twoside,leqno]{article}
\usepackage{amsmath}
\usepackage{amssymb}
\usepackage{amsxtra}
\usepackage{amscd}
\usepackage{amsthm}
\usepackage[mathscr]{eucal}
\usepackage[all]{xy}
\usepackage{color}

\theoremstyle{plain}
\newtheorem{thm}[subsection]{Theorem}
\newtheorem{prop}[subsection]{Proposition}

\newtheorem{lem}[subsection]{Lemma}

\theoremstyle{definition}

\newtheorem{rem}[subsection]{Remark}
\newtheorem{para}[subsection]{}

\newenvironment{pf}{\proof[\proofname]}{\endproof}

\numberwithin{equation}{subsection}

\begin{document}
\title
{Geometric polarized log Hodge structures for semistable families}
\author{
Taro Fujisawa and Chikara Nakayama
}

\date{}
\maketitle

\newcommand\Cal{\mathcal}
\newcommand\define{\newcommand}

\define\bG{\bold G}
\define\bZ{\bold Z}
\define\bC{\bold C}
\define\bR{\bold R}
\define\bQ{\bold Q}
\define\bN{\bold N}
\define\bP{\bold P}
\define\cl{\mathrm{cl}}%
\define\fil{\mathrm{fil}}%
\define\fl{\mathrm{fl}}%
\define\gp{\mathrm{gp}}%
\define\fs{\mathrm{fs}}%
\define\an{\mathrm{an}}%
\define\et{\mathrm{et}}%
\define\mult{\mathrm{mult}}%
\define\sat{\mathrm{sat}}%
\define\Ker{\mathrm{Ker}\,}%
\define\Coker{\mathrm{Coker}\,}%
\define\Hom{\operatorname{\mathrm{Hom}}}%
\define\Aut{\operatorname{\mathrm{Aut}}}%
\define\Mor{\operatorname{\mathrm{Mor}}}%
\define\rank{\mathrm{rank}\,}%
\define\gr{\mathrm{gr}}%
\define{\cHom}{\operatorname{\mathcal{H}\mathit{om}}}
\define{\HOM}{\cHom}
\define{\cExt}{\operatorname{\mathcal{E}\mathit{xt}}}
\define{\cMor}{\operatorname{\mathcal{M}\mathit{or}}}
\define\cB{\Cal B}
\define\cO{\Cal O}
\renewcommand{\O}{\cO}
\define\cS{\Cal S}
\define\cM{\Cal M}
\define\cG{\Cal G}
\define\cH{\Cal H}
\define\cE{\Cal E}
\define\cF{\Cal F}

\newcommand{\ep}{\varepsilon}
\newcommand{\pe}{\frak p}
\newcommand{\Spec}{\operatorname{Spec}}
\newcommand{\val}{{\mathrm{val}}}
\newcommand{\bs}{\operatorname{\backslash}}
\newcommand{\Lie}{\operatorname{Lie}}

\renewcommand{\emph}{\it}

\newcommand{\ssm}{\smallsetminus}
\newcommand{\sig}{\sigma}
\newcommand{\Sig}{\Sigma}
\newcommand{\lan}{\langle}
\newcommand{\ran}{\rangle}
\newcommand{\fg}{\frak g}
\newcommand{\cD}{\check D}
\newcommand{\G}{\Gamma}
\define\Ext{\operatorname{\mathrm{Ext}}}%
\def\SL{\operatorname{SL}}
\def\Ad{\operatorname{Ad}}
\def\fsl{{\frak s\frak l}}
\renewcommand\Im{\mathrm{Im}\,}%
\define\End{\operatorname{\mathrm{End}}}%
\newcommand{\ket}{\mathrm{k\acute{e}t}}
\newcommand{\klog}{\mathrm{klog}}


\define{\cmh}{cohomological mixed Hodge }
\define{\kos}{\operatorname{\mathrm{Kos}}}
\define{\bnnZ}{{\bold Z}_{\ge 0}}
\define{\id}{\operatorname{\mathrm{id}}}
\define{\e}{\mathbf e}
\define{\vc}{\mathbf c}
\define{\limind}{\operatorname*{\varinjlim}}
\newcommand{\RomII}{\uppercase\expandafter{\romannumeral 2}\ }
\newcommand{\RomIV}{\uppercase\expandafter{\romannumeral 4}\ }
\define{\cC}{\Cal C}
\define{\ula}{\underline{\lambda}}
\define{\dlog}{\operatorname{dlog}}
\define{\tr}{\operatorname{Tr}}
\newcommand{\coh}{\operatorname{H}}
\define{\res}{\operatorname{Res}}


\begin{abstract}
  We prove that a projective semistable morphism of fs log analytic spaces yields polarized log Hodge structures in the canonical way.
\end{abstract}
\renewcommand{\thefootnote}{\fnsymbol{footnote}}
\footnote[0]{\hspace{-18pt}2020 {\it Mathematics Subject Classification}. Primary 32S35; Secondary 14A21, 14D07\\  
{\it key words}. Hodge theory, log geometry, log Hodge structure}

{\bf Contents}

\medskip

\S\ref{s:mt}. Main theorems

\S\ref{s:pf}.  Proof of Theorem \ref{t:main}

\S\ref{s:alb}. Log Picard varieties and log Albanese varieties  

\S\ref{s:correction}.  Correction to some remarks in \cite{FN3}

\bigskip

\section*{Introduction}

  It is a fundamental problem in log Hodge theory to prove that a vertical,
saturated, projective, and log smooth morphism $X \to S$ of fs log analytic spaces yields polarized log Hodge structures.
  In the case where $S$ is log smooth over $\bC$, this was proved in \cite{KMN} by a reduction to the classical case
by restricting everything to the largest open set $S_{\mathrm{triv}}$ of $S$ where the log structure is trivial. 

  In the case of general base where $S_{\mathrm{triv}}$ can be empty, the method in \cite{KMN} does not work, and we have studied this case by a different approach 
by constructing polarized limiting mixed Hodge structures 
\`a la Steenbrink as in \cite{F}: 
  In \cite{FN2}, we treated the case where the base $S$ is the standard log point and $X$ is strict semistable over $S$, which is the first result 
for
this problem in the case where $S$ is not necessarily log smooth.
  In \cite{FN3}, we generalized it to the case where the log rank of $S$ is
less than or equal to one by reducing to \cite{FN2}. 
  In the present paper, we proved the case where $f$ is semistable over the base of higher log ranks, that is, any fiber of $f$ is semistable over the $\bN^r$-log point for some $r \ge0$, so that any $S$ with a free log structure can appear as a base. 
  This is enabled by a recent development of the theory of 
polarized limiting mixed Hodge structures 
in the higher dimensional situation (\cite{Fmultiss}).

  We state the main results in Section \ref{s:mt}. 
  The structure of the proofs of the main results described in Section \ref{s:pf} is parallel to that in \cite{FN3}. 
  An essential part is to show the constancy of the weight filtration. 
  We prove that it is a consequence of a result in \cite{Fmultiss} by checking several commutativities between log Hodge side
and limiting mixed Hodge side.
  Section \ref{s:alb} is an application. 
  Section \ref{s:correction} is a correction of some remarks in our earlier work \cite{FN3}.

\bigskip

\noindent {\sc Acknowledgments.} 
The authors thank to Kazuya Kato 
for pointing out some errors in our previous work \cite{FN3}. 
  A large part of Section \ref{s:correction} is due to him.
The first author is partially supported by
JSPS, Kakenhi (C) No.\ JP20K03542.
  The second author is partially supported by JSPS, Kakenhi (B) No.\ JP23340008, (C) No.\ JP16K05093, and (C) No.\ JP21K03199. 
The second author thanks J.\ C.\ for suggesting this work. 

\bigskip

\noindent {\sc Notation and Terminology.} 
  For the basic notions on log geometry in the analytic context, see \cite{KN} and \cite{KU}. 
  For an analytic space $S$, we denote by ${\cal O}_S$ the sheaf of 
holomorphic functions over $S$. 
  For an fs log analytic space $S$, we denote by $M_S$ the log structure on 
it. 

  For a monoid or a sheaf of monoids $M$, we denote by $M^{\times}$ the subgroup of the invertible elements and 
by $\overline M$ the quotient $M/M^{\times}$.

  A morphism $f\colon X \to S$ of fs log analytic spaces is called {\it vertical} 
if the quotient monoid of $\overline M_{X,x}$ by the image of $\overline M_{S,f(x)}$ is a group 
for any $x \in X$. 
  It is called {\it projective} 
if its underlying morphism of analytic spaces is so.

  The log structure $M$ on an analytic space is called {\it free} if each stalk of $\overline M$ is a free monoid, that is, isomorphic to $\bN^r$ for some $r\ge0$.
  We say that $r$ is the {\it log rank} at that point. 
  An fs log analytic space is called a {\it free log point} if its underlying analytic space is $\Spec \bC$ and the log structure is free. 
  A free log point of rank $r$ is called the {\it $\bN^r$-log point}.
  The $\bN$-log point is also called the {\it standard log point}.

  For simplicity, we often denote a pullback of a sheaf on a space by the same symbol 
as that for the original sheaf. 

\section{Main theorems}\label{s:mt}

  In this section we state the main results in this paper.

\begin{para}
  Let $f\colon X \to S$ be a morphism of fs log analytic spaces. 

  Recall that if $S$ is a free log point, $f$ is said to be a semistable log smooth degeneration if it is log smooth, the log structure of $X$ is free, and for any $x \in X$, the induced homomorphism $\overline M_{S,f(x)} \to \overline M_{X,x}$ of monoids is isomorphic to the product of some diagonal maps $\bN \to \bN^r; 1 \mapsto (1,1,\ldots,1)$ for various $r$ (cf.\ 
\cite{Fmultiss}
Definition 3.3). 

  We say that $f$ is {\it semistable} if $f$ is log smooth, 
the log structure of $S$ is free, and any fiber of $f$ is a semistable log smooth degeneration. 
\end{para}

\begin{prop}
  A morphism $f\colon X \to S$ of fs log analytic spaces is semistable if and only if it is vertical, saturated, and log smooth, and the log structures of $X$ and $S$ are free.
\end{prop}

  This is not used later and seen from the basic properties of saturated homomorphisms (cf.\ \cite{T}) and the following lemma. 

\begin{lem}
  A 
homomorphism $h\colon \bN^r \to \bN^s$ is integral
if and only if it is isomorphic to the product of $r$ homomorphisms 
$\bN \to \bN^{s(i)}$ with $s=\sum_i s(i)$.
\end{lem}

\begin{pf}
  Write the canonical bases $\bN^r$ and $\bN^s$ as $(e_i)_i$ and $(f_j)_j$, respectively. 
  It is enough to show that there is no $f_j$ dividing two $h(e_{i_1})$ 
and $h(e_{i_2})$ simultaneously. 
  Assume that $f_1$ divides $h(e_1)$ and $h(e_2)$. 
  Let $b_3=h(e_2)/f_1$ and $b_4=h(e_1)/f_1$. 
  Then $h(e_1)b_3=h(e_2)b_4$.
  By the integrality of $h$, we have $a_3,a_4 \in \bN^r$ 
such that $e_1a_3=e_2a_4$ and $h(a_3)$ divides $b_3$. 
  The first equality implies that $e_2$ divides $a_3$, hence, $h(e_2)$ divides $h(a_3)$. 
  Therefore, $h(e_2)$ divides $b_3=h(e_2)/f_1$, a contradiction.   
\end{pf}

  The following is the main theorem in this paper. 

\begin{thm}
\label{t:main}
  Let $f\colon X \to S$ be a projective and semistable morphism of fs log analytic spaces. 
  Let $q \in \bZ$. 

$(1)$ The triple 
$$\bigl(R^qf^{\log}_*{\bZ}, (R^qf_{\ket*}\omega^{\cdot\ge p}_{X/S,\ket})_{p \in \bZ}, (\cdot , \cdot )\bigr)$$
is a PLH on $S_{\ket}$.

$(2)$ The triple $$\bigl(R^qf^{\log}_*{\bZ}, (R^qf_{*}\omega^{\cdot\ge p}_{X/S})_{p \in \bZ}, (\cdot , \cdot )\bigr)$$
is a nonk\'et PLH on $S$. 
\end{thm}

  Here, the polarization $(\cdot , \cdot)$ determined by a choice of a relatively ample line bundle and 
the isomorphism between the lattice and the vector bundle after tensoring $\cO_S^{\klog}$ or $\cO_S^{\log}$ are given in \cite{FN3}.
  See ibid.\ Section 2.

\begin{rem}
  It is plausible that we can weaken  projectivity by log K\"ahlerness.
\end{rem}

\section{Proof of Theorem \ref{t:main}}\label{s:pf}
  We prove Theorem \ref{t:main}. 

\begin{para}
  Since by \cite[Theorem 2.5 (3)]{FN3}, the k\'et log Hodge to log de Rham spectral sequence is the pullback of the nonk\'et one, the triple in Theorem \ref{t:main} (1) is the pullback of the triple in Theorem \ref{t:main} (2).
  Hence, it is enough to prove Theorem \ref{t:main} (2). 
  In the rest of this section, pre-PLH and PLH are in the nonk\'et sense.

  By \cite{FN3}, we already know that the triple 
in Theorem \ref{t:main} (2) 
is a pre-PLH satisfying small Griffiths transversality and 
the condition 

\smallskip

\noindent $(*)$ the pullback with respect to any morphism from the standard log point to $S$ is a PLH. 

\smallskip

\noindent
In fact, \cite{FN3} 6.7, 2 says that it is a pre-PLH satisfying small Griffiths transversality.  
(Note here that the orthogonality can be checked by the pullback with respect to any morphism from the standard log point to $S$.
  Cf.\ \cite[Theorem 2.5 (2)]{FN3}.)
  Further, the condition $(*)$ follows from the fact that the theorem is valid when the base is the standard log point. 

 To prove the positivity, by \cite[Theorem 2.5 (2)]{FN3}, we may assume that the base is an fs log point, which is the $\bN^r$-log point by the assumption of the semistability. 

 We work in the framework of \cite{KU}.
  Let $t \in S^{\log}$, and let $H_0=(R^qf^{\log}_*{\bZ})_t/$(torsion). 
  Let $\langle\cdot,\cdot\rangle_0$ be the bilinear form on $H_{0,\bC}$ induced by our $(\cdot,\cdot)$. 
  Let $(h^{j,k})_{j,k}$ be the Hodge numbers of our pre-PLH.
  Then the 4-ple $\Phi_0=(q,(h^{j,k})_{j,k}, H_0, \langle\cdot,\cdot\rangle_0)$ 
  satisfies the condition in \cite[0.7.3]{KU} because our pre-PLH gives a PLH after some pullback to the standard log point. 
  Thus, by \cite[2.5.3]{KU}, we have an element $(\sigma, Z) \in \check D_{\mathrm{orb}}$, which is the image of $S$ by the period map. 
  It is  enough to prove that $(\sigma,Z)$ is a nilpotent orbit. 
  Let $F \in Z$. 
  By $(*)$, for some interior $N_0$ of $\sigma$, $(N_0, F)$ generates a nilpotent orbit. 
  By the following lemma, the rest is to see that the weight filtration with respect to an interior of $\sig$ is constant, which we will show below in \ref{constant}. 
\end{para}

\begin{lem}
  Let $\Phi_0=(q,(h^{j,k})_{j,k}, H_0, \langle\cdot,\cdot\rangle_0)$ be the data satisfying the condition {\rm \cite[0.7.3]{KU}}. 
  Let $\sig$ be a nilpotent cone, let $N_0$ be an interior of $\sig$, and let $F \in \check D$ such that $F$ satisfies the Griffiths transversality with respect to $\sig$. 
  Assume that the weight filtration with respect to an interior is constant and that $(N_0,F)$ generates a nilpotent orbit. 
  Then $(\sig,F)$ generates a nilpotent orbit. 
\end{lem}

\begin{pf}
  This is mentioned in the first and the second lines of \cite[p.505]{CKS}, and 
a consequence of \cite[Proposition (4.66)]{CKS}. 
  Cf.\ \cite[Proposition A.2]{KNU} and \cite[Remark A.2.1]{KNU}.
\end{pf}

\begin{para}
\label{constant}
We prove the constancy of the weight filtrations.
Let $S$ be the $\bN^r$-log point
and $X$ a projective and semistable fs log analytic space over $S$.
By Lemma \ref{lem:1} below,
we may work on $H_{0, \bC}$.

Let $q_i$ be the $i$-th unit vector of $\bN^r$,
regarded as an element of $M_S=\bN^r \oplus \bC^{\times}$
and $\gamma_1, \dots, \gamma_r$ the dual base of
$\pi_1(S^{\log})=\Hom(M_S/\cO_S^{\times}, \bZ)$
as in \cite[Section 2]{KU}.
By the action of $\pi_1(S^{\log})$ on $H_0$,
every $\gamma_i$ gives us an automorphism of $H_{0,\bC}$,
which is denoted by the same letter $\gamma_i$ for all $i=1, \dots, r$.
Then the cone $\sigma$ is generated
by $\log (\gamma_1), \dots, \log (\gamma_r)$ by definition.
On the other hand,
we have the natural integrable log connection
$\nabla \colon H_{\cO} \to \omega^1_S \otimes H_{\cO}$,
where $H_{\cO}=R^qf_{*}\omega^{\cdot}_{X/S}$
which is isomorphic to $H_{0, \bC}$
via the 
log Riemann--Hilbert correspondence in \cite{KN} 
(cf.~\cite{IKN} Theorem (6.3)).
By the identifications
$\omega^1_S \simeq \cO_S \otimes (M_S^{\gp}/\cO_S^{\times})
\simeq \cO_S \otimes \bZ^r$,
the composite of $\nabla$ and the $i$-th projection
$\omega_S^1 \otimes H_{\cO} \simeq \bZ^r \otimes H_{\cO} \to H_{\cO}$
is denoted by $\res_i(\nabla)$,
which is called the $i$-th residue morphism of $\nabla$.
Then we have
\begin{equation*}
(2\pi\sqrt{-1})^{-1}\log (\gamma_i)=\res_i(\nabla)
\end{equation*}
for $i=1, \dots, r$ as in \cite[Proposition 2.3.4]{KU}.
(This equality is different from the usual one by the sign
(see e.g. \cite[Proposition II.3.11]{D2}). 
This is caused by the difference of the convention
for the action of $\pi_1(S^{\log})$.
See \cite[Appendix A1]{KU}.)

In \cite{Fmultiss},
a complex of $\bC$-sheaves $A_{\bC}$ on $X$
and a quasi-isomorphism
$\theta \colon \omega^{\cdot}_{X/S} \to A_{\bC}$
are constructed. 
Thus we have an isomorphism
\begin{equation}
\label{eq:1}
H_{\cO}=R^qf_{*}\omega^{\cdot}_{X/S}
=\coh^q(X, \omega_{X/S}^{\cdot}) \to \coh^q(X, A_{\bC})
\end{equation}
induced by $\theta$.
Moreover,
a finite increasing
filtration $L$ on $A_{\bC}$
and morphisms of complexes
$\nu_1, \dots, \nu_r \colon A_{\bC} \to  A_{\bC}$
with $\nu_i(L_mA_{\bC}) \subset L_{m-2}A_{\bC}$ for all $m$ and $i$,
are constructed in \cite{Fmultiss}.
Let $N_i \colon \coh^q(X, A_{\bC}) \to \coh^q(X, A_{\bC})$
denote 
the morphism induced by $\nu_i$ for $i=1, \dots, r$.
We set
$N(\vc)=\sum_{i=1}^{r}c_iN_i$
for $\vc=(c_1, \dots, c_r) \in \bC^r$.
Then it is proved, in \cite{Fmultiss},
that the morphism $N(\vc)^k$
induces an isomorphism
$$
\gr_k^L\coh^q(X, A_{\bC}) \to \gr_{-k}^L\coh^q(X, A_{\bC})
$$
for all $k \in \bnnZ$ and $\vc \in (\bR_{> 0})^r$.
Thus we know that $L$ coincides with
the weight filtration of $N(\vc)$ for all $\vc \in (\bR_{>0})^r$.
Therefore it suffices to prove
\begin{equation*}
\res_i(\nabla)=-N_i
\end{equation*}
for all $i=1, \dots, r$
under the isomorphism \eqref{eq:1}.
So, we have to describe
the connection $\nabla$
in terms of $N_1, \dots, N_r$,
via the identification \eqref{eq:1}.
This can be done by a similar argument
to \cite[(4.22)]{S76} and \cite[Section 5]{F2}.

First, we note that the connection $\nabla$ 
on $H_{\cO}=\coh^q(X, \omega^{\cdot}_{X/S})$
is induced from the exact sequence
\begin{equation*}
0 \to \omega^1_S \otimes \omega_{X/S}^{\cdot}[-1]
\to \omega_X^{\cdot}
/\omega^2_S \wedge \omega_X^{\cdot}[-2]
\to \omega_{X/S}^{\cdot} \to 0
\end{equation*}
as the connecting homomorphism
\begin{equation*}
\coh^q(X, \omega^{\cdot}_{X/S})
\to \omega^1_S \otimes \coh^q(X, \omega^{\cdot}_{X/S}),
\end{equation*}
where
$\omega^2_S \wedge \omega_X^{\cdot}[-2]$
is the subcomplex of $\omega_X^{\cdot}$
consisting of $\omega^2_S \wedge \omega_X^{p-2} \subset \omega_X^p$
for all $p$.

In \cite[Definition 3.21]{Fmultiss},
the complex $A_{\bC}$
is defined by
\begin{gather*}
A_{\bC}
=\Bigl(\bC[u_1, \dots, u_r] \otimes_{\bC} \omega^{\cdot}_X
\Bigr/\sum_{i=1}^{r}W(i)_{-1}\Bigr)[r],
\end{gather*}
where $W(i)_{-1}$ is a subcomplex
of $\bC[u_1, \dots, u_r] \otimes_{\bC} \omega^{\cdot}_X$
for $i=1, \dots, r$
defined in Definition 3.21 of \cite{Fmultiss}.
Moreover, the morphisms
$\theta \colon \omega^{\cdot}_{X/S} \to A_{\bC}$
and $\nu_i \colon A_{\bC} \to A_{\bC}$ ($i=1, \dots, r$)
are defined in \cite[Definitions 3.15 and 4.6]{Fmultiss} by
\begin{gather*}
\theta=1 \otimes \dlog t_1 \wedge \cdots \wedge \dlog t_r \wedge, \\
\nu_i=(u_i \cdot) \otimes \id,
\end{gather*}
where
$1 \in \bC[u_1, \dots, u_r]$ and
$(u_i \cdot)$ denotes the $\bC$-linear map
$\bC[u_1, \dots, u_r] \to \bC[u_1, \dots, u_r]$
by multiplying $u_i$. 
We define a morphism
$\mu \colon A_{\bC} \to \omega^1_S \otimes A_{\bC}$ by
\begin{equation}
\label{eq:2}
\mu=\sum_{i=1}^{r}\dlog t_i \otimes \nu_i
\end{equation}
and a complex
$B_{\bC}$ by
$B_{\bC}^p=(\omega^1_S \otimes A_{\bC}^{p-1}) \oplus A_{\bC}^p$
with the differential
\begin{equation*}
(\omega^1_S \otimes A_{\bC}^{p-1}) \oplus A_{\bC}^p
\ni (x, y) \mapsto (-(\id \otimes d)(x)-\mu(y), dy)
\in (\omega^1_S \otimes A_{\bC}^p) \oplus A_{\bC}^{p+1}
\end{equation*}
for all $p$.
By using the fact that $\nu_i$ is a morphism of complexes for all $i$,
we can easily check that $\mu$ is a morphism of complexes
and $B_{\bC}$ is a complex. 
Then there exists a natural exact sequence
\begin{equation*}
0 \to \omega_S^1 \otimes A_{\bC}[-1]
\to B_{\bC} \to A_{\bC} \to 0
\end{equation*}
from the definition of $B_{\bC}$.
For $i=1, \dots, r$,
a morphism of $\cO_X$-modules 
$\eta_i \colon \omega^p_X \to A_{\bC}^{p-1}$
is defined by
\begin{equation*}
\eta_i(\omega)
=(-1)^{r-i}
1 \otimes \dlog t_1 \wedge \cdots \wedge \dlog t_{i-1}
\wedge \dlog t_{i+1} \wedge \cdots \wedge \dlog t_r \wedge \omega
\end{equation*}
for $\omega \in \omega^p_X$.
Then
$\eta_i$ trivially
factors through the surjection
$\omega^p_X \to \omega^p_X/\omega^2_S \wedge \omega^{p-2}_X$ 
and the equality
\begin{equation*}
d\eta_i(\omega)=-\eta_i(d\omega)-\nu_i(\theta(\overline{\omega}))
\end{equation*}
holds for $i=1, \dots, r$,
where $\overline{\omega}$ is the image of $\omega$
by the surjection $\omega^p_X \to \omega^p_{X/S}$. 
Thus we obtain a morphism
\begin{equation*}
\eta=(\sum_{i=1}^{r}\dlog t_i \otimes \eta_i)+\theta
\colon \omega^{\cdot}_X/\omega^2_S \wedge \omega^{\cdot}_X[-2]
\to B_{\bC},
\end{equation*}
which turns out to be a morphism of complexes
by direct calculation.
Moreover, we have a 
commutative diagram
\begin{equation*}
\begin{CD}
0 @>>> \omega^1_S \otimes \omega^{\cdot}_{X/S}[-1]
@>>> \omega^{\cdot}_X/\omega^2_S \wedge \omega^{\cdot}_X[-2]
@>>> \omega^{\cdot}_{X/S} @>>> 0 \\
@. @V{\id \otimes \theta}VV @VV{\eta}V @VV{\theta}V \\
0 @>>> \omega_S^1 \otimes A_{\bC}[-1]
@>>> B_{\bC} @>>> A_{\bC} @>>> 0
\end{CD}
\end{equation*}
with exact rows. 
Therefore the connection $\nabla$
is identified with the morphism induced by $-\mu$,
and hence $\res_i(\nabla)=-N_i$ for all $i=1, \dots, r$
via the isomorphism \eqref{eq:1} as desired.

\begin{lem}
\label{lem:1}
Let $V$ be a finite dimensional $\bR$-vector space
and $N_1, N_2$ nilpotent endomorphisms of $V$.
Then $\id \otimes N_1$ and $\id \otimes N_2$
are nilpotent endomorphisms
of $\bC$-vector space $V_{\bC}=\bC \otimes_{\bR} V$.
If $W(\id \otimes N_1)=W(\id \otimes N_2)$ on $V_{\bC}$,
then $W(N_1)=W(N_2)$ on $V$.
\end{lem}
\begin{proof}
By the very definition of the weight filtration
of a nilpotent endomorphism,
we have $W(\id \otimes N_i)_m=\bC \otimes_{\bR} W(N_i)_m$
for all $m$ and for $i=1,2$.
Therefore we have
\begin{equation*}
W(N_1)_m=W(\id \otimes N_1)_m \cap V
=W(\id \otimes N_2)_m \cap V=W(N_2)_m
\end{equation*}
for all $m$.
\end{proof}

\end{para}

\section{Log Picard varieties and log Albanese varieties}\label{s:alb}

  Exactly in the same way as in \cite[Section 7]{FN3}, we can construct the log Picard variety and the log Albanese variety in the semistable case. 

  We briefly explain the results.  
  The proofs are all parallel to there. 

\begin{thm}
\label{t:good}
  Let $f\colon X \to S$ be a projective and semistable morphism of fs log analytic spaces. 
  Then $f$ is good both in the sense of {\rm \cite{KKN}} $7.1$ and in the sense of {\rm \cite{KKN}} $7.2$. 
  The quasi-isomorphisms in the condition $\rm{(i)}$ in {\rm \cite{KKN}} $7.1$ and in the condition $\rm{(i)}$ in {\rm \cite{KKN}} $7.2$ coincide. 
\end{thm}

\begin{rem}
  Here we use the corrected definition of the goodness explained in \cite{FN2} 2.4. 
\end{rem}

\begin{para}
  The proof of Theorem \ref{t:good} is the same as that of \cite[Theorem 7.1]{FN3} except that 
 we use Theorem \ref{t:main} (2) instead of \cite[Corollary 2.7]{FN3}.
  Note that the proof there shows that 
in general, a vertical, saturated, projective, and log smooth morphism of fs log analytic spaces satisfies the conditions of the goodness in \cite[7.1]{KKN} or 
\cite[7.2]{KKN} (the isomorphisms in both coincide) except the latter half of the condition \cite[7.1 (iii)]{KKN}. 
\end{para}

\begin{para}
  By Theorem \ref{t:good}, the log Picard variety $A^*_{X/S}$ and the log Albanese variety $A_{X/S}$ are well-defined (cf.\ \cite[7.2]{FN3}).

  As in the same way as in \cite[Section 7]{FN3}, the conclusions of all results in Sections 8 and 10 in \cite{KKN} are valid for such an $f$.  
  In particular, 
the conclusions of Theorem 7.3, Corollary 7.4, and Proposition 7.5
of \cite{FN3} are valid for a projective and semistable morphism as in Theorem \ref{t:good}. 
\end{para}

\section{Correction to some remarks in \cite{FN3}}
\label{s:correction}

  In this section, we correct some errors in \cite{FN3}, which concern 
some remarks after all proofs of the main results in \cite{FN3} and 
do not affect the proofs in this paper.
  The pointing out \ref{incorrect} with the suggestion of the validity of Proposition \ref{p:blowup} and the counter example \ref{counter} are due to K.\ Kato.
  
\begin{para}
  First we overview the corrections. 
  Below in this section, pre-PLH and PLH are in the nonk\'et sense. 
  In the latter half of \cite[6.7]{FN3}, we give two remarks on a pre-PLH on a general base, which roughly say: 

(1) The positivity cannot be checked by the pullbacks to the standard log point;

(2) The positivity cannot be checked by the pullback to a log blow-up;

\noindent and we gave an 
example for both. 

  But, the 
example is incorrect as we explain below in \ref{incorrect}. 
  In fact, (2) is not valid, that is, the positivity can be checked by the pullback to a log blow-up, as we see in Proposition \ref{p:blowup} below. 
  The statement (1) is still valid and we describe a correct 
example below in \ref{counter}. 
\end{para}

\begin{para}
\label{incorrect}
  We explain how the example in the latter half in \cite[6.7]{FN3} is wrong. 
  In fact, there are some confusions about what are the morphisms from the standard log point to the $\bN^2$-log point. 
  Let $S=(\Spec \bC, \bN^2 \oplus \bC^{\times})$ as in there and let 
$S_0=(\Spec \bC, \bN \oplus \bC^{\times})$.
  Then, it is true that the example there satisfies the following property: 
It is not a PLH but the pullback of it with respect to the morphism $S_0 \to S$ charted by any local homomorphism $\bN^2 \to \bN$ is a PLH. 
  In this sense, it is a correct counter example for the analogue of \cite[Proposition 6.6 (1)]{FN3} about the positivity. 
  But, in fact, there are more morphisms $S_0 \to S$, which are not necessarily charted. 
  So, in the context there, what we had to study was not the analogue of 
\cite[Proposition 6.6 (1)]{FN3} but the following question: 

\smallskip

\noindent $(**)$ If the positivity is satisfied by the pullback with respect to any morphism $S_0 \to S$ (not necessarily coming from a homomorphism $\bN^2 \to \bN$), does it imply that the positivity is satisfied by the original pre-PLH?

\smallskip

  We explain that the example we gave is not a counter example for $(**)$. 
  Consider the morphism $S_0 \to S$ determined by the homomorphism $\bN^2 \to \bN \oplus \bC^{\times}$ sending $e_1$ to $(1,1)$ and $e_2$ to $(1,e^{-2 \pi})$. 
 ($(e_j)_{j=1,2}$ is the canonical basis as in there.) 
  Then, the pullback to $S_0$ with respect to this morphism is not a PLH because $y-(y+1)+1=0$ for any $y$.

  Similarly, the pullback to the log blow-up $S'$ 
(the notation is as in there) is not a PLH at the point of the exceptional fiber 
which induces the above $S_0 \to S$
so that it is not a counter example 
for the analogue of \cite[Proposition 6.6 (2)]{FN3} about the positivity, too. 
\end{para}

The analogue of \cite[Proposition 6.6 (2)]{FN3} about the positivity is valid as follows. 
(We slightly generalize the statement by replacing ``log blow-up'' with ``log modification''
(\cite[Definition 3.6.12]{KU})).

\begin{prop}
\label{p:blowup}
  Let $H$ be a nonk\'et pre-PLH of weight $w$ on an fs log analytic space $S$. 
  Assume that $H$ satisfies the small Griffiths transversality. 
  Let $S'\to S$ be a log modification 
  Assume that the pullback of $H$ on $S'$ satisfies the positivity (i.e., satisfies the condition {\rm \cite[1.2 (2)]{FN3}} after pulling back to each point of $S'$).
  Then $H$ also satisfies the positivity. 
  (In other words, if the pullback on $S'$ is a 
nonk\'et PLH, then $H$ is a 
nonk\'et PLH.)
\end{prop}

\begin{pf}
  We may and will assume that $S$ is an fs log point. 
  Let $s$ be its unique point. 
  Let 
$\tilde S:=\{h \in \Hom(M_{S,s}, \bC^{\mathrm{mult}})\,|\,h$ induces the inclusion $\bC^{\times} \to \bC\}$
 be the canonical toric variety associated to $S$ (cf.\ \cite[1.11.1]{KNU6}), where $\bC^{\mathrm{mult}}$ means $\bC$ regarded as a monoid by the multiplication. 
  We identify $S$ with the origin of $\tilde S$. 
  We extend $H$ to a pre-PLH $\tilde H$ over $\tilde S$ canonically (cf.\ 
\cite[4.1.4]{KNU6}). 
  Since $H$ satisfies the Griffiths transversality, $\tilde H$ also satisfies it. 

Take a chart of $\tilde S$ by $\overline M_{S,s}$. 
  By the definition of the log modification, there is a finite rational subdivision $\Sigma$ 
of $\Hom(\overline M_{S,s},\bR_{\ge0})$
such that $S'=S(\Sigma)$. 
  Let $\tilde S'=\tilde S(\Sigma)$
 and we regard $S'$ as 
the special fiber of the morphism $\tilde S' \to \tilde S$ over the origin of $\tilde S$. 
  By assumption, the positivity for the pullback of $\tilde H$ on $\tilde S'$ is satisfied at each point of $S'$. 
  Since the positivity for a pre-PLH satisfying the small Griffiths transversality 
is an open condition (\cite[Proposition 2.1]{KNUmixpure}), the pullback of $\tilde H$ is a PLH on an open neighborhood of $S'$. 
  Since $\tilde S'\to \tilde S$ is proper and $\tilde S'_{\mathrm{triv}} \to \tilde S_{\mathrm{triv}}$ is an isomorphism, this implies that $\tilde H$ satisfies the positivity in the intersection of $\tilde S_{\mathrm{triv}}$ and an open neighborhood of the origin $S$. 
  Hence $H$ is a PLH 
because it is a PH after sufficiently shifts (cf.\ \cite[4.1.4]{KNU6}).
\end{pf}

  As a byproduct of the above proof, we see the following fact. 

\begin{prop}
  Let $S$ be an fs log analytic space and $s$ a point of $S$.
  Then any PLH $H$ on $s$ extends to a PLH on a neighborhood of $s$ in $S$. 
\end{prop}

\begin{pf}
  By taking a chart, we may and will assume that $S$ is the canonical toric variety associated to $s$.
  Then $H$ canonically extends to a pre-PLH on $S$ satisfying the Griffiths transversality. 
  Since the positivity is an open condition, it is satisfied on a neighborhood of $s$.
  Hence the extension is a PLH on a neighborhood of $s$. 
\end{pf}

\begin{para}
\label{counter}
We describe a 
correct
 counter example due to K.\ Kato for $(**)$.
Let $H_{\bZ}$ be the free abelian group of rank 6 generated by $e_1, e_2, ...,e_6$. 
  Endow it with the symmetric bilinear form $\langle \cdot, \cdot\rangle$ defined by that for $1 \le j \le k\le 6$, 
$\langle e_j, e_k\rangle$=1 if $(j,k)=(1,5), (2,6)$,
$\langle e_j, e_k\rangle=-1$ if $(j,k)=(3,3), (4,4)$
and $\langle e_j, e_k\rangle=0$ otherwise.
  Let the Hodge filtration $F$ be as 
$F^2=\{0\}$, $F^1=\bC e_5 + \bC e_6$, $F^0 = \bC e_3+\bC e_4+\bC e_5+\bC e_6$, and $F^{-1}=\bC \otimes H_{\bZ}$.
  Let $N_1, N_2$ be the nilpotent endomorphisms on $H_{\bZ}$ whose matrix representations with respect to the canonical base are 
$$N_1=
\begin{bmatrix}
0&0&1&0&0&0 \\
0&0&0&1&0&0 \\
0&0&0&0&1&0 \\
0&0&0&0&0&1 \\
0&0&0&0&0&0 \\
0&0&0&0&0&0
\end{bmatrix}, 
N_2=
\begin{bmatrix}
0&0&1&1&0&0 \\
0&0&1&-1&0&0 \\
0&0&0&0&1&1 \\
0&0&0&0&1&-1 \\
0&0&0&0&0&0 \\
0&0&0&0&0&0
\end{bmatrix},$$
respectively. 

  Then these data generate a pre-PLH of weight 0 over the $\bN^2$-log point whose Hodge numbers are 
$h^{j,k}=2$ if $(j,k)=(-1,1), (0,0), (1,-1)$
and $h^{j,k}=0$ otherwise,
and which satisfies the Griffiths transversality. 

  Let $a, b$ be positive real numbers. 
  Then the weight filtration with respect to $N=aN_1+bN_2$ is characterized as follows: 
$$ \gr^W_{-2}H_{\bQ}=\bQ e_1 + \bQ e_2, \quad
\gr^W_{0}H_{\bQ}=\bQ e_3 + \bQ e_4, \quad 
\gr^W_{2}H_{\bQ}=\bQ e_5 + \bQ e_6, $$
unless $a^2-2b^2=0$.
  But, if $a^2-2b^2=0$, the rank of $N$ is $2$ so that the weight filtration with respect to $N$ is different from the above.
  Hence this pre-PLH is not a PLH because the weight filtration with respect to an interior $N$ of the monodromy cone is not constant. 

  We prove that the pullback with respect to any morphism from the standard log point is a PLH. 
  It is enough to show that for any positive integers $a,b$ and for any real number $c$, the pair $(N, \exp(icN_1)F)$ generates a nilpotent orbit. 
  By \cite[Proposition (4.66)]{CKS}
(applied after an appropriate Tate twist),
it suffices to show that $(W(N), \exp(icN_1)F)$ gives a mixed Hodge structure polarized by $N$. 
  This is reduced to the positivity of 
$\begin{bmatrix} a+b & b \\ b & a-b
\end{bmatrix}^2$.
\end{para}

\noindent Taro Fujisawa

\noindent Tokyo Denki University \\
5 Senju Asahi, Adachi, Tokyo 120-8551 \\ Japan

\noindent fujisawa@mail.dendai.ac.jp

\bigskip

\noindent Chikara Nakayama

\noindent Department of Economics \\ Hitotsubashi University \\
2-1 Naka, Kunitachi, Tokyo 186-8601 \\ Japan

\noindent c.nakayama@r.hit-u.ac.jp
\end{document}